\documentclass[a4,12pt,reqno]{amsart}
\usepackage{amssymb,amstext,amscd,amsfonts,mathdots,mathrsfs}
\numberwithin{equation}{section}
\usepackage{hyperref}% for ref
\usepackage{setspace}% for spacing
\usepackage[all]{xy}% for quiver
\usepackage{amsthm}% for proof
\usepackage{enumerate}% for [(1)] or [(a)]
\usepackage{diagbox}% for diag
\usepackage{indentfirst}%for indent
\usepackage{color}% for color
\usepackage{colortbl}
\usepackage{amsmath}% for the equation label
\usepackage{tikz}% for tikz quiver
\usetikzlibrary{shapes.geometric, arrows}
\usetikzlibrary{calc}
\usepackage{multirow,multicol}
\def\lcm{\operatorname{lcm}}
\def\rad{\operatorname{rad}}
\def\op{\operatorname{op}}
\def\End{\operatorname{End}}
\newtheorem{theorem}{Theorem}[section]
\newtheorem{theorem*}{Theorem}
\newtheorem{corollary}[theorem]{Corollary}
\newtheorem{corollary*}[theorem*]{Corollary}
\newtheorem{lemma}[theorem]{Lemma}
\newtheorem{proposition}[theorem]{Proposition}

\theoremstyle{definition}
\newtheorem{definition}[theorem]{Definition}

\newtheorem*{question*}{Question}

\newtheorem*{conjecture*}{Conjecture}

\newtheorem*{notation*}{Notation}

\newtheorem*{claim*}{Claim}

\begin{document}
\setlength{\baselineskip}{16pt}

\title
[Representation-finite tensor product algebras]
{Representation-finite tensor product algebras}

\author[Q. Wang]{Qi Wang}
\address{Yau Mathematical Sciences Center, Tsinghua University, Beijing 100084, China.}
\email{infinite-wang@outlook.com}
%%%%%%%%%%%%%%%%%%%%%%%%%%%%%%%%%%%%%%%%%%%%%%%%%%%%%%%%%%%%%%%%%%%%%%%%%%%
%\date{\today}
%%%%%%%%%%%%%%%%%%%%%%%%%%%%%%%%%%%%%%%%%%%%%%%%%%%%%%%%%%%%%%%%%%%%%%%%%%%

\begin{abstract}
In this paper, we complete the classification of representation-finite tensor product algebras in terms of quiver with relations.
\end{abstract}
%%%%%%%%%%%%%%%%%%%%%%%%%%%%%%%%%%%%%%%%%%%%%%%%%%%%%%%%%%%%%%%%%%%%%%%%%%%
\maketitle

\section{Introduction}
The representation type of algebra is a fundamental concept that describes the complexity of the module category for a given algebra, and therefore often serves as a necessary first step in the study of specific classes of algebras. It is well-established by Drozd in \cite{Dr-tame-wild} that the representation type of any finite-dimensional algebra $A$ over an algebraically closed field $k$ is precisely one of the following: representation-finite, (infinite-)tame, wild.

The study of the representation type of tensor product algebras has a long history, as evidenced by many works such as \cite{AR-triangular-matrix, BD-finite-group, S-nakayama, HM-selfinjective, L-special-alg}. A systematic approach to classifying tensor product algebras by their representation type can be attributed to Leszczy$\acute{\text{n}}$ski and Skowro$\acute{\text{n}}$ski. Their series of papers \cite{L-rep-type, LS-tame-triangular-matrix, LS-tame-tensor-product} provide a thorough analysis and complete description of non-wild tensor product algebras. Moreover, their work together with \cite{H-group-alg, BD-finite-group, MS-rep-finite-group, MS-rep-finite-group-correction, S-polynomial-growth-group} have distinguished representation-finite cases from non-wild cases for several classes of algebras, including group algebras $AG$ of finite groups $G$ over $A$, $n\times n$ triangular matrix algebras $T_n(A)$ over $A$, and enveloping algebras $A\otimes A^{\op}$ of $A$. Besides, the number of indecomposable modules over a representation-finite enveloping algebra $A\otimes A^{\op}$ could be explicitly calculated, see \cite{MZ20}. However, it remains unresolved to classify representation-finite cases from non-wild cases for general tensor product algebras. 

This paper aims to address and resolve the above open problem. Since $A\otimes k\simeq A$ for any algebra $A$, we focus on tensor products of non-simple algebras. It is established in \cite[Proposition 1.8]{L-rep-type} that the tensor product $A\otimes B\otimes C$ is representation-infinite for any three non-simple algebras $A, B, C$. Additionally, if $A\otimes B$ is representation-finite, then both $A$ and $B$ are representation-finite. Consequently, it remains to investigate the representation-finiteness of $A\otimes B$, for any two representation-finite non-simple algebras $A$ and $B$. 

We first review the results presented in \cite{LS-tame-triangular-matrix} concerning the representation-finiteness of $2\times 2$ triangular matrix algebras $T_2(B)$. Set $\mathbf{A}_2:=k(\xymatrix@C0.7cm{1\ar[r]&2})$. Then, we have $T_2(B):=T_2(k)\otimes B\simeq\mathbf{A}_2\otimes B$. The Galois covering technique plays a crucial role in the classification of these algebras. We refer to Subsection \ref{subsection-covering} for the construction of Galois coverings.
\begin{theorem*}[{\cite[Theorem 4]{LS-tame-triangular-matrix}}]\label{theorem-1}
Let $B$ be a representation-finite non-simple algebra. Then, $\mathbf{A}_2\otimes B$ is representation-finite if and only if the Galois coverings of $B$ and $B^{\op}$ do not contain one of the algebras in the family (IT) in \cite{LS-tame-triangular-matrix} as a quotient algebra.
\end{theorem*}

Miyamoto and the author \cite{MW21} have completely determined the representation-finiteness of $A\otimes B$, where both $A$ and $B$ are representation-finite simply connected algebras with at least three simple modules (i.e., $|A|, |B|\ge 3$). Note that a representation-finite simply connected algebra $A$ with $|A|=2$ is exactly $\mathbf{A}_2$. We refer to \cite{Assem-simply-connected} for an in-depth survey of simply connected algebras, and refer to \cite{BG-standard form} for the relation between the Galois covering technique and simply connected algebra. Let $\mathbb{A}_n$ be the unoriented Dynkin diagram with $n$ vertices. Set $N(n):=k\overset{\rightarrow}{\mathbb{A}}_n/(\rad^2 k\overset{\rightarrow}{\mathbb{A}}_n)$ with
$$
\overset{\rightarrow}{\mathbb{A}}_n : \xymatrix@C=0.7cm@R=0.5cm{
1\ar[r]&2\ar[r]&3\ar[r]&4\ar[r]&\cdots\ar[r]& n-1\ar[r]& n}.
$$
We then review the main results of \cite{MW21} as follows.

\begin{theorem*}[{\cite[Theorem 3.11, Theorem 4.8, Theorem 4.16]{MW21}}]\label{theorem-2}
Let $A$ and $B$ be two representation-finite simply connected algebras with $|A|, |B|\ge 3$. Then, $A\otimes B$ is representation-finite if and only if $A\simeq N(n)$ for $n\ge 3$, and one of the following holds.
\begin{enumerate}
\item[\rm{(1)}] $B\simeq N(m)$, for $m\ge 3$.

\item[\rm{(2)}] $B$ is a Nakayama algebra with $\rad^2 B\neq 0$ such that 
\begin{itemize}
\item $n=3$, $\rad^3 B=0$ and $B$ does not contain one of 
\begin{equation}\label{B_1}
\mathbf{B}_1:=k\left(
\xymatrix@C=1cm@R=0.5cm{
\circ\ar[r]^-\alpha&\circ\ar[r]^-\beta&\circ\ar[r]^-\gamma&\circ\ar[r]^-\delta&\circ}
\right)
\bigg/
\langle \alpha\beta\gamma, \beta\gamma\delta \rangle,
\end{equation}
\begin{equation}\label{B_2}
\mathbf{B}_2:=k\left(
\xymatrix@C=0.75cm@R=0.5cm{
\circ\ar[r]^-\alpha&\circ\ar[r]^-\beta&\circ\ar[r]^-\gamma&\circ\ar[r]^-\delta&\circ\ar[r]^-\sigma &\circ}
\right)
\bigg/
\langle \alpha\beta\gamma, \gamma\delta \rangle,
\end{equation}
and their opposite algebras as a quotient algebra.

\item $n\ge 4$ and $B$ does not contain 
\begin{equation}\label{B_3}
\mathbf{B}_3:=k\left(
\xymatrix@C=1cm@R=0.5cm{
\circ\ar[r]^-\alpha&\circ\ar[r]^-\beta&\circ\ar[r]^-\gamma&\circ}
\right)
\bigg/
\langle\alpha\beta\gamma\rangle
\end{equation}
as a quotient algebra.
\end{itemize}

\item[\rm{(3)}] $B$ is not a Nakayama algebra such that
\begin{itemize}
\item $B$ is isomorphic to the path algebra of $\xymatrix@C=0.7cm{\circ\ar[r]&\circ&\circ\ar[l]}$ or $\xymatrix@C=0.7cm{\circ&\circ\ar[r]\ar[l]&\circ}$.

\item $B$ is isomorphic to $\mathbf{B}_5$ or $\mathbf{B}_5^{\op}$, where
\begin{equation}\label{B_5}
\mathbf{B}_5:=k\left(
\xymatrix@C=1cm@R=0.5cm{
\circ\ar[r]^-\alpha&\circ&\circ\ar[l]_-\beta&\circ\ar[l]_-\gamma}
\right)
\bigg/
\langle \gamma\beta \rangle.
\end{equation}

\item $n=3$, $|B|\ge 5$ and $B$ does not contain one of the path algebras of type $\mathbb{A}_4$,
\begin{equation}\label{B_6}
\mathbf{B}_6:=k\left(
\xymatrix@C=1cm@R=0.5cm{
\circ\ar[r]^-\alpha&\circ&\circ\ar[l]_-\beta&\circ\ar[l]_-\gamma\ar[r]^-\delta&\circ}
\right)
\bigg/
\langle \gamma\beta \rangle,
\end{equation}
\begin{equation}
\mathbf{B}_7:=k\left(
\xymatrix@C=0.85cm@R=0.5cm{
\circ\ar[r]^-\delta &\circ\ar[r]^-\alpha&\circ&\circ\ar[l]_-\beta&\circ\ar[l]_-\gamma}
\right)
\bigg/
\langle \gamma\beta, \delta\alpha \rangle,
\end{equation}
and their opposite algebras as a quotient algebra.
\end{itemize}
\end{enumerate}
\end{theorem*}

In this paper, we address the remaining cases of $A\otimes B$ where at least one of the algebras $A$ and $B$ is not simply connected. Set $N^\circ(n):=k\Delta_n/(\rad^2 k\Delta_n)$ with
$$
\Delta_n : \xymatrix@C=0.8cm@R=0.5cm{
1\ar[r]_-{\alpha}&2\ar[r]_-{\beta}&3\ar[r]_-{\gamma}&4\ar[r]_-{\delta}&5\ar[r]_-{\sigma}&\cdots\ar[r]& n\ar@/_1.5pc/[llllll]^{\ }}.
$$
We present the following main results of this paper:
\begin{theorem*}\label{theorem-3}
Let $A$, $B$ be representation-finite algebras such that $A$ is not simply connected and $B\not\simeq \mathbf{A}_2$. Then, $A\otimes B$ is representation-finite if and only if one of the following holds.
\begin{enumerate}
\item[\rm{(1)}] $A\simeq N^\circ(n)$ for $n\ge 1$ and $B$ is isomorphic to one of  $\mathbf{B}_5$, $\mathbf{B}_5^{\op}$ and the path algebras of type $\mathbb{A}_3$, or $B$ is a Nakayama algebra having no $\mathbf{B}_3$ as a quotient algebra.

\item[\rm{(2)}] $B\simeq N(m)$ for $m\ge 3$ and the quiver of $A$ is $\Delta_n$ such that
\begin{itemize}
\item $A$ is isomorphic one of the algebras $k\Delta_2/\langle \beta\alpha \rangle$, $k\Delta_3/\langle \beta\gamma, \gamma\alpha \rangle$,  $k\Delta_4/\langle \beta\gamma, \delta\alpha \rangle$, $k\Delta_4/\langle \beta\gamma, \gamma\delta, \delta\alpha \rangle$, $k\Delta_5/\langle \beta\gamma, \gamma\delta, \sigma\alpha \rangle$, $k\Delta_5/\langle \beta\gamma, \gamma\delta, \delta\sigma, \sigma\alpha \rangle$.

\item $m=3$ and $A$ is isomorphic to $k\Delta_4/\langle \alpha\beta\gamma, \gamma\delta, \delta\alpha \rangle$ or $k\Delta_5/\langle \alpha\beta\gamma, \gamma\delta, \delta\sigma, \sigma\alpha \rangle$.

\item $m=3$, $|A|\ge 6$, $\rad^3 A=0$ and $A$ does not contain $\mathbf{B}_1$ or $\mathbf{B}_2$ as a quotient.

\item $m\ge 4$, $|A|\ge 6$ and $A$ does not contain $\mathbf{B}_3$ as a quotient algebra.
\end{itemize}
\end{enumerate}
\end{theorem*}
\begin{proof}
See Theorem \ref{theo::local-A}, Theorem \ref{theo::NO(2)-A}, Proposition \ref{prop::NO(n)-A} and Theorem \ref{theo::NO(n)-N(n)}.
\end{proof}

This paper is organized as follows. In Section 2, we review some necessary definitions and introduce the Galois covering technique as required for our analysis. In Section 3, we present our main reduction theorems, which simplify the general problem to smaller cases. Sections 4, 5, and 6 are devoted to proving Theorem \ref{theorem-3} by dividing it into the cases where $|A|=1$, $|A|=2$, and $|A|\ge 3$, respectively.

\section{Preliminaries}
Let $A=kQ/I$ be a bound quiver algebra over an algebraically closed field $k$ with a finite connected quiver $Q=(Q_0, Q_1, s, t)$ and an admissible ideal $I$ of $kQ$. Then, $A$ is called a \emph{local} algebra if $|Q_0|=1$; otherwise, $A$ is called a \emph{non-local} algebra. A vertex $i\in Q_0$ is a \emph{source} (resp., \emph{sink}) if there is no arrow $\alpha\in Q_1$ with $t(\alpha)=i$ (resp., $s(\alpha)=i$). A path $\omega$ in $kQ$ of length $\ell\ge 1$ is called an \emph{orientated cycle} if $s(\omega)=t(\omega)$, and moreover, an orientated cycle $\omega$ is called a \emph{loop} if $\ell=1$. Then, $A$ is said to be \emph{triangular} if $Q$ admits no loops and oriented cycles. A relation $\rho=\sum_{i=1}^{m}\lambda_i\omega_i$ in $I$ with $0\neq \lambda_i\in k$ is a $k$-linear combination of paths $\omega_i$ of length at least 2 having the same source and target. If $m=2$, then $\omega_1-\omega_2$ is called a \emph{commutativity relation} in the sense that we always use $\rho=0$ to indicate $\rho\in I$. We refer to \cite{ASS} for more background on the representation theory of quivers.

We denote by $\rad A$ the Jacobson radical of $A$ and by $A^{\op}$ the opposite algebra of $A$. Then, $A$ is said to be \emph{radical square zero} if $\rad^2 A=0$. 

Let $A=kQ_A/I_A$ and $B=kQ_B/I_B$ be two bound quiver algebras. We define 
\begin{itemize}
\item $Q_A\otimes Q_B$: the quiver consisting of vertex set $(Q_A\otimes Q_B)_0:=(Q_A)_0\times (Q_B)_0$
and arrow set $(Q_A\otimes Q_B)_1:=((Q_A)_1\times (Q_B)_0)\cup ((Q_A)_0\times (Q_B)_1)$, together with the source map $s(-)$ and the target map $t(-)$ defined as
\begin{align*}
&s(\alpha\times j)=s_A(\alpha)\times j, \quad  s(i\times \beta)=i\times s_B(\beta), \\
&t(\alpha\times j)=t_A(\alpha)\times j, \quad t(i\times \beta)=i\times t_B(\beta),
\end{align*}
for any $(\alpha,j)\in (Q_A)_1\times (Q_B)_0$ and $(i,\beta)\in (Q_A)_0\times (Q_B)_1$.

\item $I_A\diamond I_B$: the two-sided ideal of $k(Q_{A}\otimes Q_B)$ generated by elements in $((Q_A)_0 \times I_B)\cup (I_A\times (Q_B)_0)$ and the elements being of the form $(i,\beta_{st})(\alpha_{ij},t)-(\alpha_{ij},s)(j,\beta_{st})$, where $\alpha_{ij}$ and $\beta_{st}$ run through all arrows $\alpha_{ij}:i\rightarrow j$ in $(Q_A)_1$ and $\beta_{st}:s\rightarrow t$ in $(Q_B)_1$.
\end{itemize}

Let $A\otimes B$ be the tensor product algebra of $A$ and $B$, that is, the $k$-algebra defined by the multiplication $(a_1\otimes b_1)\cdot(a_2\otimes b_2)=a_1a_2\otimes b_1 b_2$, for any $a_1, a_2\in A$ and $b_1, b_2\in B$. Then, we have $A\otimes B\simeq k(Q_A\otimes Q_B)/ (I_A\diamond I_B)$ as a bound quiver algebra (see \cite[Lemma 1.3]{L-rep-type}). We review some basic properties of tensor product algebras as follows.

\begin{proposition}[{\cite[Lemma 1.2]{L-rep-type}}]
Let $A$, $B$ and $C$ be bound quiver algebras. Then,
\begin{enumerate}
\item[\rm{(1)}] $A\otimes B \simeq B\otimes A$.
\item[\rm{(2)}] $(A\otimes B)^{\op} \simeq A^{\op} \otimes B^{\op}$.
\item[\rm{(3)}] $A\otimes (B \otimes C)\simeq (A\otimes B)\otimes C$.
\item[\rm{(4)}] $A\otimes B$ is connected if and only if both $A$ and $B$ are connected.
\end{enumerate}
\end{proposition}

\subsection{Galois covering}\label{subsection-covering}
We reformulate the constructions in \cite{N-covering-monomial} as follows.

Let $Q=(Q_0, Q_1)$ be a finite connected quiver. We denote by $Q(a,b)$ the set of arrows in $Q$ starting from $a$ and ending at $b$. A \emph{covering} of $Q$ is a connected quiver $\widetilde{Q}$ together with a map $\mathcal{F}: \widetilde{Q}\rightarrow Q$ mapping vertices $\widetilde{a},\widetilde{b},\ldots$ to vertices $a, b, \ldots$, arrows $\widetilde{\alpha}, \widetilde{\beta},\ldots$ to arrows $\alpha, \beta, \ldots$, and paths to paths (i.e., $\alpha\beta$ is a path in $Q$ if $\widetilde{\alpha}\widetilde{\beta}$ is a path in $\widetilde{Q}$), such that 
\begin{itemize}
\item the restriction $\mathcal{F}|_0: \widetilde{Q}_0\rightarrow Q_0$ is a surjection,

\item $\mathcal{F}$ induces a bijection 
$$
\bigcup_{\mathcal{F}(\widetilde{c})=b} \widetilde{Q}(\widetilde{a}, \widetilde{c})\longrightarrow Q(a, b)
$$
for any vertex $\widetilde{a}$ of $\widetilde{Q}$ and any vertex $b$ of $Q$,

\item $\mathcal{F}$ induces a bijection 
$$
\bigcup_{\mathcal{F}(\widetilde{c})=a} \widetilde{Q}(\widetilde{c}, \widetilde{b})\longrightarrow Q(a, b)
$$
for any vertex $\widetilde{b}$ of $\widetilde{Q}$ and any vertex $a$ of $Q$.
\end{itemize}
Then, the map $\mathcal{F}: \widetilde{Q}\rightarrow Q$ naturally gives a $k$-linear homomorphism from $k\widetilde{Q}$ to $kQ$, which is also denoted by $\mathcal{F}$.

Let $A=kQ/I$ be a bound quiver algebra where $I$ is generated by a set $\Omega$ of relations in $kQ$. We denote by $\widetilde{\Omega}$ the set of all relations $\widetilde{\omega}$ in $k\widetilde{Q}$ satisfying $\mathcal{F}(\widetilde{\omega})\in \Omega$. Let $\widetilde{I}$ be the two-sided ideal of $k\widetilde{Q}$ generated by $\widetilde{\Omega}$, and set $\widetilde{A}:=k\widetilde{Q}/\widetilde{I}$. Then, the map $\mathcal{F}: \widetilde{Q}\rightarrow Q$ induces the \emph{universal Galois covering} of $A$, i.e.,
$$
\mathcal{G}: \widetilde{A}\longrightarrow \widetilde{A}/G=A,
$$
where $G=\Pi_1(Q, I)$ is the fundamental group of $A$, see \cite[(1.8)]{S06} for more details. 

Note that the algebra $\widetilde{A}$ may be infinite-dimensional. For example, the universal Galois covering $\widetilde{B}$ of the local algebra $B=k[x]/(x^n)$ for $n\ge 2$ is given by an infinite linear quiver 
$$
\xymatrix@C=0.7cm@R=0.5cm{
\cdots \ar[r]^-x&\circ\ar[r]^-x&\circ \ar[r]^-x&\circ \ar[r]^-x&\cdots}
$$
bounded by $x^n=0$, and the fundamental group of $B$ is $\mathbb{Z}$. Let $\widetilde{e}_a$ be the idempotent of $\widetilde{A}$ corresponding to the vertex $\widetilde{a}$ in $\widetilde{Q}$. We say that $\widetilde{A}$ is \emph{locally representation-finite} if there are only finitely many indecomposable $\widetilde{A}$-modules $M$ satisfying $M\widetilde{e}_a\neq 0$ for each vertex $\widetilde{a}\in \widetilde{Q}$, up to isomorphism.

The most critical property of Galois coverings is as follows. This property is fully used in Section 4 and Section 5.

\begin{theorem}[{\cite[Lemma 3.3 and Theorem 3.6]{G-universal-cover}, see also \cite{MP-universal-cover}}]\label{theo::covering}
Let $A$ be a bound quiver algebra and $\widetilde{A}$ its universal Galois covering. Then, $A$ is representation-finite if and only if $\widetilde{A}$ is locally representation-finite.
\end{theorem}

We may apply Theorem \ref{theo::covering} to tensor product algebras. Let $\mathcal{G}^A: \widetilde{A}\rightarrow \widetilde{A}/G_A=A$ be the universal Galois covering of $A$ with group $G_A$ and  $\mathcal{G}^B: \widetilde{B}\rightarrow \widetilde{B}/G_B=B$ the universal Galois covering of $B$ with group $G_B$. According to \cite[Lemma 1.7]{L-rep-type},
$$
\mathcal{G}^A\otimes \mathcal{G}^B: \widetilde{A}\otimes \widetilde{B} \rightarrow \widetilde{A}\otimes \widetilde{B}/(G_A\times G_B)=A\otimes B
$$
is exactly the universal Galois covering of $A\otimes B$ with group $G_A\times G_B$.

\begin{corollary}\label{cor::covering}
Let $\widetilde{A}\otimes \widetilde{B}$ be the universal Galois covering of $A\otimes B$. Then, $A\otimes B$ is representation-finite if and only if $\widetilde{A}\otimes \widetilde{B}$ is locally representation-infinite.
\end{corollary}

A triangular algebra $A$ is said to be \emph{simply connected} (see \cite[Definition 1.2]{AS}) if the fundamental group $\Pi_1(Q,I)$ is trivial for any bound quiver presentation $kQ/I$ of $A$. All tree algebras are simply connected. It is given in \cite[Lemma 1.7]{L-rep-type} that $A\otimes B$ is simply connected if and only if both $A$ and $B$ are simply connected.

A representation-finite algebra $A$ admits a \emph{standard form}
%\footnote{A modern definition of standard algebra is given in \old{[S06, (1.8)]}, saying that an algebra $A$ is standard if it has a simply connected universal covering. However, the definition in \cite{BG-covering} means that $A$ is standard if the Auslander-Reiten quiver of $A$ is a mesh-category.}
$\overline{A}$ such that $A$ and $\overline{A}$ share the same Auslander-Reiten quiver (see \cite[Corollary 5.2]{BG-covering}), and $\overline{A}$ admits a simply connected universal Galois covering
$$
\mathcal{G}: \widetilde{\overline{A}}\longrightarrow \widetilde{\overline{A}}/G=\overline{A},
$$
where $G=\Pi_1(Q_A, I_A)$ is the fundamental group of $A$, see \cite[Section 3]{BG-standard form}. If moreover, $A$ is a simply connected representation-finite algebra, then $A$ is standard (\cite{BG-covering}) and hence, 
$$
\widetilde{\overline{A}}=\overline{A}=A.
$$
We refer to \cite{DS-covering} for a study on the covering theory of representation-infinite algebras.

\subsection{Some algebras}
Let $\mathbf{A}_n^\epsilon$ ($n\geq 2$) be the path algebra of $\mathbb{A}_n$ associated with orientation $\epsilon$, where $\mathbb{A}_n$ is the Dynkin diagram of the form:
$$
\xymatrix@C=0.7cm@R=0.5cm{
1\ar@{-}[r]&2\ar@{-}[r]&3\ar@{-}[r]&4\ar@{-}[r]&\cdots\ar@{-}[r]& n-1\ar@{-}[r]& n},
$$
and the orientation $\epsilon:=(\epsilon_1,\epsilon_2,\ldots,\epsilon_{n-1})$ is defined by $\epsilon_i=+$ if $i\rightarrow i+1$, $\epsilon_i=-$ if $i+1\rightarrow i$. In particular, we denote $\mathbf{A}_2^{(+)}\simeq \mathbf{A}_2^{(-)}$ by $\mathbf{A}_2$ for short.

Recall that a module $M$ is said to be \emph{uniserial} if it has a unique composition series. Then, an algebra $A$ is said to be \emph{Nakayama} if all indecomposable projective $A$-modules and all indecomposable injective $A$-modules are uniserial. We review a simple characterization of Nakayama algebras as follows.
\begin{proposition}[{\cite[V. Theorem 3.2]{ASS}}]\label{prop::Nakayama}
A bound quiver algebra $A=kQ/I$ is Nakayama if and only if $Q$ is one of the following quivers:
$$
\overset{\rightarrow}{\mathbb{A}}_n : \xymatrix@C=0.7cm@R=0.5cm{
1\ar[r]&2\ar[r]&3\ar[r]&4\ar[r]&\cdots\ar[r]& n-1\ar[r]& n},
$$
$$
\Delta_n : \xymatrix@C=0.7cm@R=0.5cm{
1\ar[r]&2\ar[r]&3\ar[r]&4\ar[r]&\cdots\ar[r]& n-1\ar[r]& n\ar@/_1.5pc/[llllll]^{\ }},
$$
or equivalently, if and only if $Q$ does not contain one of $\xymatrix@C=0.6cm{\circ&\circ\ar[l]\ar[r]&\circ}$ and $\xymatrix@C=0.6cm{\circ\ar[r]&\circ &\circ\ar[l]}$ as a subquiver.
\end{proposition}

\begin{definition}
The Nakayama algebras with radical square zero are given as
$$
N(n):=k\overset{\rightarrow}{\mathbb{A}}_n/(\rad^2 k\overset{\rightarrow}{\mathbb{A}}_n)
\quad \text{and} \quad
N^\circ(n):=k\Delta_n/(\rad^2 k\Delta_n).
$$
\end{definition}

Note that $N(2)\simeq \mathbf{A}_2$ and $N^\circ(1)\simeq k[x]/(x^2)$. One observation is that $A$ with $|A|\ge 3$ is not a Nakayama algebra with radical square zero if and only if $A$ contains $\mathbf{A}_3^\epsilon$ for some $\epsilon$ as a quotient algebra. Another observation is that all $N^\circ(n)$'s share the same universal Galois covering given by
$$
\xymatrix@C=0.7cm@R=0.5cm{
\cdots \ar[r]^-x&\circ\ar[r]^-x&\circ \ar[r]^-x&\circ \ar[r]^-x&\cdots}
$$
with $x^2=0$; this coincides with the universal Galois covering of $k[x]/(x^2)$.

Let $A:=kQ$ be a hereditary algebra. Recall from \cite{HR-tilted} that an $A$-module $T$ is said to be a \emph{tilting module} if $|T|=|A|$, $\text{Ext}_A^1(T,T)=0$ and the projective dimension of $T$ is at most one. If $T$ is a tilting $A$-module, we call the endomorphism algebra $B:=\End_A T$ a \emph{tilted algebra} of type $Q$. If moreover, $T$ is contained in a preprojective component of the Auslander-Reiten quiver of $A$, we call $B$ a \emph{concealed algebra} of type $Q$. 

Tame concealed algebras are precisely the concealed algebras of type $\widetilde{\mathbb{A}}_n$, $\widetilde{\mathbb{D}}_n (n\geqslant 4)$, $\widetilde{\mathbb{E}}_6$, $\widetilde{\mathbb{E}}_7$ and $\widetilde{\mathbb{E}}_8$. It is shown in \cite[(1.2)]{AS} that tame concealed algebras are simply connected. It is also worth mentioning that tame concealed algebras have been classified completely by quiver and relations in \cite{HV-tame-concealed} (see also \cite{B-critical}).

\section{Reduction theorems}
In this section, we review some fundamental reduction methods, and explain our main reduction strategy for later use. First of all, the following proposition gives us the reason why it suffices to consider the representation-finiteness of $A\otimes B$.
\begin{proposition}[{\cite[Proposition 1.8]{L-rep-type}}]\label{prop::three-tensor}
If $A$, $B$ and $C$ are three non-simple algebras, then $A\otimes B\otimes C$ is representation-infinite. Moreover, $A\otimes B\otimes C$ is tame if and only if $A\simeq B\simeq C\simeq \mathbf{A}_2$.
\end{proposition}

We then list some well-known reduction methods on the representation-finiteness of algebras, especially, tensor product algebras, which are due to the existence of some full faithful functors or algebra isomorphisms.
\begin{proposition}\label{prop::quotient}
If $A$ is a representation-finite algebra, then
\begin{enumerate}
\item[\rm{(1)}] the quotient algebra $A/I$ is representation-finite for any two-sided ideal $I$ of $A$.
\item[\rm{(2)}] the idempotent truncation $eAe$ is representation-finite for any idempotent $e$ of $A$.
\end{enumerate}
\end{proposition}

\begin{corollary}
If $A\otimes B$ is representation-finite, both $A$ and $B$ are representation-finite.
\end{corollary}

\begin{proposition}
The opposite algebra $A^{\op}$ shares the same representation type with $A$.
\end{proposition}

\begin{proposition}[{\cite[Theorem 3.11]{MW21}}]\label{prop::3*3}
Suppose $A$ and $B$ contain $\mathbf{A}_3^\epsilon$ and $\mathbf{A}_3^\omega$ as quotient algebras, for some choices of $\epsilon$ and $\omega$, respectively. Then, $A\otimes B$ is representation-infinite.
\end{proposition}

We need the handy criteria for representation-finiteness of radical square zero algebras. Recall that the separated quiver $Q^s=(Q_0^s, Q_1^s)$ of a connected quiver $Q=(Q_0,Q_1)$ is given by $Q_0^s:=\{i^+,i^-\mid i\in Q_0\}$ and $Q_1^s:=\{i^+\rightarrow j^-\mid i\rightarrow j \in Q_1 \}$.

\begin{proposition}[{\cite{G-separated, ARS}}]\label{prop::rad-square-zero}
Let $A=kQ/\rad^2 kQ$. Then, $A$ is representation-finite if and only if the separated quiver $Q^s$ of $Q$ is a disjoint union of Dynkin quivers.
\end{proposition}

For example, the local algebra $k[x,y]/(x^2,y^2,xy)$ is representation-infinite. 

\begin{corollary}\label{coro::double-arrow}
If $A\otimes B$ is representation-finite, both $A$ and $B$ admit no double arrow and no double loops on one vertex.
\end{corollary}
\begin{proof}
It follows that the Kronecker algebra $k(\xymatrix@C=0.7cm{1\ar@<0.5ex>[r]^{ }\ar@<-0.5ex>[r]_{ }&2})$ is representation-infinite.
\end{proof}

We are now able to give our main reduction methods on $A\otimes B$ as follows.
\begin{lemma}\label{lem::cycle}
Suppose both $A$ and $B$ contain a loop or an oriented cycle in their quivers. Then, $A\otimes B$ is representation-infinite.
\end{lemma}
\begin{proof}
(1) If both $A$ and $B$ contain a loop, then $A\otimes B$ contains $k[x]/(x^2)\otimes k[x]/(x^2)\simeq k[x,y]/(x^2, y^2)$ as a quotient algebra, which is a tame symmetric special biserial algebra (see \cite[III, Theorem 1]{Er-tame-block}). Hence, $A\otimes B$ is representation-infinite by Proposition \ref{prop::quotient}.

(2) If $A$ contains a loop and $B$ contains an oriented cycle, say,
\begin{center}
$\Delta_n=\xymatrix{1 \ar[r]_-x & 2 \ar[r]_-x & \cdots \ar[r]_-x & n \ar@/_1pc/[lll]_-x}$, $n\ge 2$,
\end{center}
then $A$ admits $A':=k[x]/(x^2)$ as a quotient and $B$ admits $B':=k\Delta_n/\langle x^2\rangle$ as a quotient.
By definition, the quiver of $A'\otimes B'$ is given by adding one loop on each vertex in $\Delta_n$. Then, we have $i^+\rightarrow i^-$, $i^+\rightarrow (i+1)^-$ for any $i\in \{1,2,\ldots, n\}$ $(n+1:=1)$ such that the separated quiver of the quiver of $A'\otimes B'$ is given as a quiver of type $\widetilde{\mathbb{A}}_{2n-1}$ with zigzag orientation. Hence, $A'\otimes B'$ is representation-infinite by Proposition \ref{prop::rad-square-zero}, and so is $A\otimes B$.

(3) If both $A$ and $B$ contain an oriented cycle, say, $\Delta_m$ and $\Delta_n$ ($m, n\ge 2$), respectively, then the quiver $Q_\Lambda$ of $\Lambda:=k\Delta_m/\langle x^2\rangle\otimes k\Delta_n/\langle x^2\rangle$ can be displayed in an $n\times m$ table. We use $\{(a,b)\mid 1\le a\le n, 1\le b\le m\}$ to indicate the vertex set of $Q_\Lambda$. Set $n+1:=1$ and $m+1:=1$. Then, we have $(a,b)^+\rightarrow (a+1,b)^-$ and $(a,b)^+\rightarrow (a,b+1)^-$, for any $1\le a\le n$, $1\le b\le m$. It gives the following subquiver $Q_\Lambda'$ of $Q_\Lambda^s$ with zigzag orientation
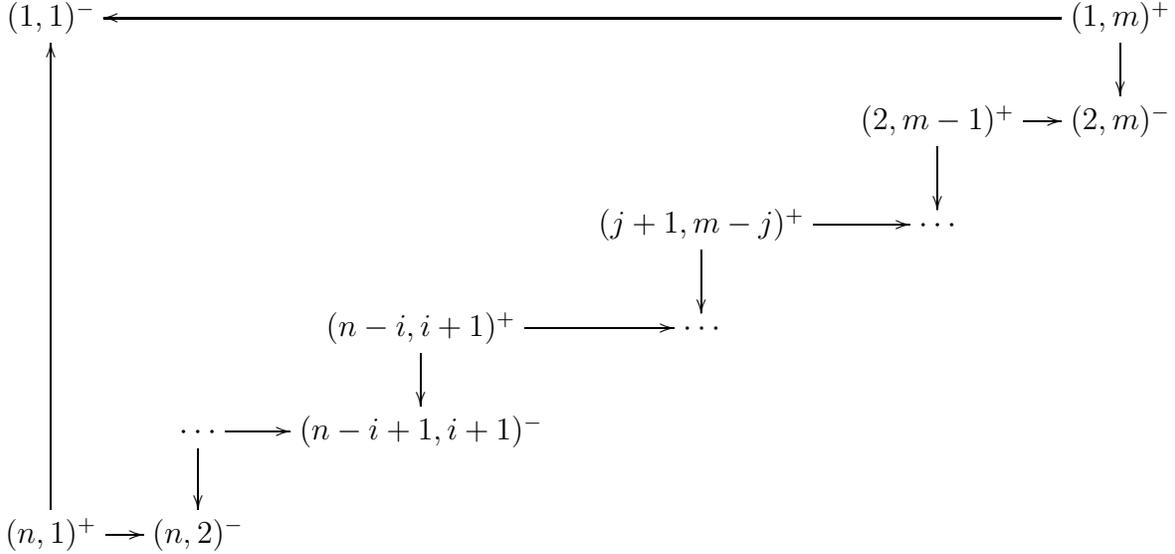
\begin{center}
$\xymatrix@C=0.5cm@R=0.7cm{
(1,1)^-&&&&&(1,m)^+\ar[lllll]\ar[d]\\
&&&&(2,m-1)^+\ar[r]\ar[d]&(2,m)^-\\
&&&(j+1,m-j)^+\ar[r]\ar[d]&\cdots&\\
&&(n-i,i+1)^+\ar[d]\ar[r]&\cdots&&\\
&\cdots\ar[d]\ar[r] &(n-i+1,i+1)^-&&&\\
(n,1)^+\ar[uuuuu]\ar[r]&(n,2)^-&&&&
}$
\end{center}
for some integers $i,j \ge 0$. Each pair $(i,j)$ satisfying $i+j+1=\lcm(m,n)$ gives 
\begin{center}
$n-i=j+1 \mod n \quad $ and $\quad i+1=m-j \mod m$
\end{center}
such that $(n-i,i+1)^+$ and $(j+1,m-j)^+$ identify the same vertex in $Q_\Lambda'$. Moreover, the subquiver $Q_\Lambda'$ has $2i+2j+1+1=2\lcm(m,n)$ vertices. Then, the separated quiver $Q_\Lambda^s$ of $Q_\Lambda$ admits a full subquiver of type $\widetilde{\mathbb{A}}_{2\lcm(m,n)-1}$ with zigzag orientation. Hence, $\Lambda$ is representation-infinite by Proposition \ref{prop::rad-square-zero}, and so is $A\otimes B$.
\end{proof}

The strategy in the proof of Lemma \ref{lem::cycle} can be used to prove the following theorem.

\begin{theorem}\label{theo::no-cycle}
Suppose both the underlying graphs of the quivers of $A$ and $B$ contain a cycle. Then, $A\otimes B$ is representation-infinite.
\end{theorem}
\begin{proof}
Using Corollary \ref{coro::double-arrow}, we may assume that both $Q_A$ and $Q_B$ contain no double arrow. We assume that a subquiver $\Delta_A$ of $Q_A$ and a subquiver $\Delta_B$ of $Q_B$ induce the cycles in the underlying graphs of $Q_A$ and $Q_B$, respectively. If both $\Delta_A$ and $\Delta_B$ admit no sink, then $\Delta_A$ and $\Delta_B$ are loops or orientated cycles in $A$ and $B$, and $A\otimes B$ is representation-infinite following Lemma \ref{lem::cycle}. We then assume in the following that $\Delta_A$ admits at least one sink, and hence, $|\Delta_A|=m\ge 3$.

(1) Suppose $|\Delta_B|=1$, i.e., $\Delta_B$ is a loop in $Q_B$. The separated quiver of $\Delta_A\otimes \Delta_B$ can be displayed in the form 
\begin{center}
$
\vcenter{\xymatrix@C=0.5cm@R=0.7cm{
1^+\ar@<-0.5ex>[d]&2^-&3^+\ar@<-0.5ex>[d]&\cdots&2i^-&(2i+1)^+\ar@<-0.5ex>[d]&\cdots\\
1^-&2^+\ar@<0.4ex>[u]&3^-&\cdots&2i^+\ar@<0.4ex>[u]&(2i+1)^-&\cdots
}}
$,
\end{center}
where the orientation of each neighbor $(j,j+1)$ in $\Delta_A$ gives either $j^+\rightarrow (j+1)^-$ or $(j+1)^+\rightarrow j^-$, for any $1\le j\le m$ and $m+1:=1$. We then obtain a full subquiver of type $\widetilde{\mathbb{A}}_{2m-1}$ with zigzag orientation, and $A\otimes B$ is representation-infinite by Proposition \ref{prop::rad-square-zero}.

(2) Suppose $|\Delta_B|=2$. 
We consider the following table
\begin{center}
$
\vcenter{\xymatrix@C=0.5cm@R=0.7cm{
11^+\ar@<-0.5ex>[d]&12^-&\cdots&1m^\epsilon&11^-&12^+\ar@<-0.5ex>[d]&\cdots&1m^{-\epsilon}\\
21^-&22^+\ar@<0.4ex>[u]&\cdots&2m^{-\epsilon}\ar@{-}[u]&21^+\ar@<0.4ex>[u]&22^-&\cdots&2m^{\epsilon\ar@{-}@<0.4ex>[u]}
}}
$,
\end{center}
with $\epsilon=+$ if $m$ is odd and $\epsilon=-$ if $m$ is even. Then, $j\rightarrow j+1$ in $\Delta_A$ gives $1j^+\rightarrow 1(j+1)^-$ and $2j^+\rightarrow 2(j+1)^-$, $j+1\rightarrow j$ in $\Delta_A$ gives $1(j+1)^+\rightarrow 1j^-$ and $2(j+1)^+\rightarrow 2j^-$, for any $1\le j\le m$ and $m+1:=1$. Then, it is always possible to find a full subquiver of type $\widetilde{\mathbb{A}}_n$ with zigzag orientation in the separated quiver of $\Delta_A\otimes \Delta_B$, for some $n\in \mathbb{N}$, we omit the details. It turns out that $A\otimes B$ is representation-infinite by Proposition \ref{prop::rad-square-zero}.

(3) Suppose $|\Delta_B|\ge 3$. If $\Delta_B$ admits at least one sink, then $\Delta_A\otimes \Delta_B$ contains the following subquiver
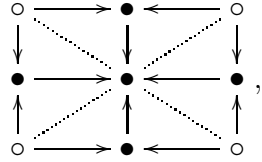
\begin{center}
$\vcenter{\xymatrix@C=1cm@R=0.5cm{\circ \ar[r] \ar[d]\ar@{.}[dr]&\bullet\ar[d] &\circ \ar@{.}[dl] \ar[l]\ar[d]\\ \bullet \ar[r] &\bullet &\bullet\ar[l] \\ \circ\ar[r]\ar[u]\ar@{.}[ur]&\bullet \ar[u]&\circ   \ar@{.}[ul] \ar[l]\ar[u]}}$,
\end{center}
where the bullets indicate a tame concealed algebra of type $\widetilde{\mathbb{D}}_4$.  

Otherwise, $\Delta_B$ is an orientated cycle in $Q_B$. Since $\Delta_A$ admits at least one sink, it has at least one source too. More precisely, $\Delta_A$ contains a subquiver of the form
\begin{center}
$\xymatrix{\circ\ar[r] &\circ  &\circ \ar[l]&\cdots\ar[l]&\circ\ar[l] &\circ\ar[l]\ar[r]&\circ }$.
\end{center}
Then, the separated quiver of $\Delta_A\otimes \Delta_B$ contains a subquiver which is of the form
\begin{center}
$\vcenter{\xymatrix@C=0.7cm@R=0.5cm{
&+\ar[d]&&&\\
+\ar[r]&-&+\ar[l]\ar[d]&&\\
&&\cdots&+\ar[l]\ar[r]\ar[d]&-\\
&&&-&
}}$,
\end{center}
where the middle dots contain only zigzag orientation. It suffices to use Proposition \ref{prop::rad-square-zero} to deduce that $A\otimes B$ is representation-infinite.
\end{proof}

\begin{theorem}\label{theo::B-type-D}
Suppose $A\not\simeq \mathbf{A}_2$. If the underlying graph of the quiver $Q_B$ of $B$ contains
$$
\mathbb{D}_4: \vcenter{\xymatrix@C=0.5cm@R=0.1cm{
&&1\ar@{-}[dl]\\
4&3\ar@{-}[l]&\\
&&2\ar@{-}[ul]}}
$$
as a subgraph, then $A\otimes B$ is representation-infinite.
\end{theorem}
\begin{proof}
Let $\mathbf{D}_4^\epsilon$ be the path algebra of $\mathbb{D}_4$ associated with orientation $\epsilon$, where $\epsilon:=(\epsilon_1,\epsilon_2,\epsilon_3)$ is defined by $\epsilon_1=+$ if $1\rightarrow 3$, $\epsilon_1=-$ if $3\rightarrow 1$, $\epsilon_2=+$ if $2\rightarrow 3$, $\epsilon_2=-$ if $3\rightarrow 2$, $\epsilon_3=+$ if $3\rightarrow 4$, $\epsilon_3=-$ if $4\rightarrow 3$. It suffices to show that $A\otimes (\mathbf{D}_4^\epsilon/\rad^2 \mathbf{D}_4^\epsilon)$ is representation-infinite, for $\epsilon=(+,+,+), (+,+,-)$, up to isomorphism. 

If $|A|=1$ and $A$ is non-simple, then $A$ admits $A':= k[x]/(x^2)$ as a quotient algebra. If $|A|=2$, $A\not \simeq \mathbf{A}_2$ and $A$ is representation-finite, then $A$ admits 
$$
A':=k(\xymatrix{\circ\ar@<0.5ex>[r]^-{x} & \circ \ar@<0.5ex>[l]^-{x}})/\langle x^2\rangle
$$
as a quotient algebra. In both two cases, the separated quiver of the quiver of $A'\otimes (\mathbf{D}_4^{(+++)}/\rad^2 \mathbf{D}_4^{(+++)})$ contains a subquiver which is of the form 
$$
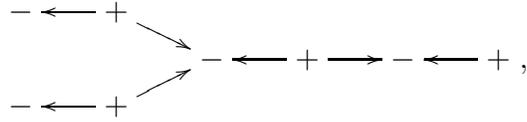
\vcenter{\xymatrix@C=0.7cm@R=0.1cm{
-&+\ar[l]\ar[dr]&&&&\\
&&-&+\ar[l]\ar[r]&-&+\ar[l]\\
-&+\ar[l]\ar[ur]&&&&
}},
$$
and the separated quiver of the quiver of $A'\otimes \mathbf{D}_4^{(++-)}$ is displayed as 
$$
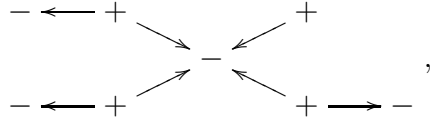
\vcenter{\xymatrix@C=0.7cm@R=0.1cm{
-&+\ar[l]\ar[dr]&&+\ar[dl]&\\
&&-&&\\
-&+\ar[l]\ar[ur]&&+\ar[lu]\ar[r]&-
}},
$$
both are not Dynkin quivers. By Proposition \ref{prop::rad-square-zero}, $A\otimes (\mathbf{D}_4^\epsilon/\rad^2 \mathbf{D}_4^\epsilon)$ is representation-infinite if $|A|=1$ or $2$.

Suppose $|A|\ge 3$. If $A$ admits $\mathbf{A}_3^{(+-)}$ or $\mathbf{A}_3^{(-+)}$ as a quotient, then $A\otimes (\mathbf{D}_4^\epsilon/\rad^2 \mathbf{D}_4^\epsilon)$ is representation-infinite following Proposition \ref{prop::3*3}. Otherwise, there is a surjection $A \twoheadrightarrow N(3)$, and it is shown in \cite[Proposition 4.10]{MW21} that $N(3)\otimes B$ is representation-infinite.
\end{proof}

\section{Local algebras}
We consider the case $A\otimes B$ with $A$ being a non-simple local algebra. It is well-known (e.g., \cite[I, 4.3.1]{Er-tame-block}) that a non-simple local algebra is representation-finite if and only if it is isomorphic to $k[x]/(x^n)$, for some $2\le n\in \mathbb{N}$. We therefore suppose in this section that
$$
A=k[x]/(x^n)\simeq k\left(\xymatrix{\circ \ar@(dr,ur)_{x}}\right) / \langle x^n\rangle.
$$
Note that $N^\circ(1)\simeq k[x]/(x^2)$. The universal Galois covering $\widetilde{A}$ of $A$ is given by
$$
\xymatrix@C=0.7cm@R=0.5cm{
\cdots \ar[r]^-x&\circ\ar[r]^-x&\circ \ar[r]^-x&\circ \ar[r]^-x&\cdots}
$$
with $x^n=0$. Using Theorem \ref{theo::no-cycle} and Theorem \ref{theo::B-type-D}, we may assume that the underlying graph of the quiver $Q_B$ of $B$ is of type $\mathbb{A}$, and $|B|\ge 2$.

\begin{proposition}\label{prop::local-A2}
Suppose $A\simeq k[x]/(x^n)$ for some $n\ge 2$, and $B\simeq \mathbf{A}_2$. Then, $A\otimes B$ is representation-finite if $n=2,3$, tame if $n=4$, and wild if $n\ge 5$.
\end{proposition}
\begin{proof}
Since the quiver of $A\otimes B$ is displayed as 
$$
\xymatrix@C=0.8cm{\circ \ar@(dl,ul)^{\ }\ar[r]^{\ }& \circ\ar@(ur,dr)^{\ }},
$$
the statement follows the classification of two-point algebras in \cite{BG-covering} and \cite{HM-two-point}.
\end{proof}

We mention that Proposition \ref{prop::local-A2} is compatible with the classification of $2\times 2$ triangular matrix algebras in \cite{LS-tame-triangular-matrix}. In fact, the universal Galois covering $\widetilde{A}\otimes \widetilde{B}$ is given by 
$$
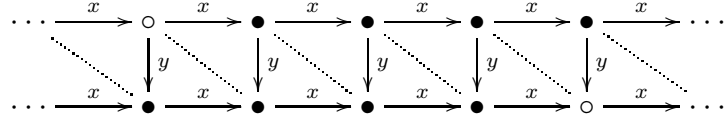
\xymatrix@C=1cm@R=0.7cm{
\cdots \ar[r]^-x\ar@{.}[dr]&\circ\ar[r]^-x\ar[d]^-y\ar@{.}[dr]&\bullet \ar[r]^-x\ar[d]^-y\ar@{.}[dr]&\bullet \ar[r]^-x\ar[d]^-y\ar@{.}[dr]&\bullet \ar[r]^-x\ar[d]^-y\ar@{.}[dr]&\bullet \ar[r]^-x\ar[d]^-y\ar@{.}[dr]&\cdots\\
\cdots \ar[r]^-x&\bullet\ar[r]^-x&\bullet\ar[r]^-x&\bullet\ar[r]^-x&\bullet\ar[r]^-x&\circ \ar[r]^-x&\cdots
}
$$
with $x^n=0$, $xy=yx$. If $n\ge 4$, the bullets in the above quiver give a tame concealed algebra of type $\widetilde{\mathbb{E}}_7$ such that $\widetilde{A}\otimes \widetilde{B}$ is not locally representation-finite. We then deduce that $A\otimes B$ is representation-infinite for any $n\ge 4$, by Corollary \ref{cor::covering}.

\begin{proposition}\label{prop::local-A3}
Suppose $A\simeq k[x]/(x^n)$ for some $n\ge 2$, and $B\simeq \mathbf{A}_3^{(++)}$ or $\mathbf{A}_3^{(+-)}$. Then, $A\otimes B$ is representation-finite if and only if $n=2$.
\end{proposition}
\begin{proof}
We notice that $\widetilde{B}=B$ since $B$ is simply connected. If $n=2$, then each idempotent truncation of $\widetilde{A}\otimes B$ is isomorphic to some idempotent truncation of $N(s)\otimes B$, for some $s\ge 2$. It is shown in \cite[Theorem 6.1]{LS-tame-tensor-product} (resp., \cite[Lemma 3.4]{MW21}) that $N(s)\otimes \mathbf{A}_3^{(++)}$ (resp., $N(s)\otimes \mathbf{A}_3^{(+-)}$) is representation-finite, it turns out that $\widetilde{A}\otimes B$ is locally representation-finite. If $n\ge 3$, then $\widetilde{A}$ contains $\mathbf{A}_3^{(++)}$ as a quotient such that $\widetilde{A}\otimes B$ is not locally representation-finite by Proposition \ref{prop::3*3}. We then obtain the statement by Corollary \ref{cor::covering}.
\end{proof}

\begin{proposition}\label{prop::local-N(m)}
Suppose $A\simeq k[x]/(x^n)$ for some $n\ge 2$, and $B\simeq N(m)$ for some $m\ge 3$. Then, $A\otimes B$ is representation-finite if and only if $n=2$.
\end{proposition}
\begin{proof}
Since $B$ is simply connected, we have $\widetilde{B}=B$. 
If $n=2$, then each non-simple idempotent truncation of $\widetilde{A}\otimes B$ is isomorphic to $N(s)$ or $N(s)\otimes N(t)$ for some $s, t\ge 2$ and $t\le |B|$. Since $N(s)\otimes N(t)$ is a representation-finite special biserial algebra shown in \cite{ABM}, we conclude that $\widetilde{A}\otimes B$ is locally representation-finite. Using Corollary \ref{cor::covering}, $A\otimes B$ is representation-finite.

If $n\ge 3$, then $\widetilde{A}$ contains
$$
\mathbf{B}_1=k\left(
\xymatrix@C=1cm@R=0.5cm{
\circ\ar[r]^-\alpha&\circ\ar[r]^-\beta&\circ\ar[r]^-\gamma&\circ\ar[r]^-\delta&\circ}
\right)
\bigg/
\langle \alpha\beta\gamma, \beta\gamma\delta \rangle
$$
as a quotient algebra. It follows from Theorem \ref{theorem-2} (2) that $\mathbf{B}_1\otimes N(3)$ is representation-infinite, then $\widetilde{A}\otimes B$ is not locally representation-finite. Hence, $A\otimes B$ is representation-infinite following Corollary \ref{cor::covering}.
\end{proof}

\begin{theorem}\label{theo::local-A}
Let $A\simeq k[x]/(x^n)$ with some $n\ge 2$. Suppose $|B|\ge 2$ and the quiver of $B$ is of type $\mathbb{A}$. Then, $A\otimes B$ is representation-finite if and only if one of the following holds.
\begin{enumerate}
\item[\rm{(1)}] $n=2,3$ and $B\simeq \mathbf{A}_2$.
\item[\rm{(2)}] $n=2$ and 
\begin{itemize}
\item $B$ is isomorphic to one of $\mathbf{A}_3^{(++)}$, $\mathbf{A}_3^{(+-)}$, $\mathbf{A}_3^{(-+)}$, $\mathbf{B}_5$, $\mathbf{B}_5^{\op}$, where
$$
\mathbf{B}_5=k\left(
\xymatrix@C=1cm@R=0.5cm{
\circ\ar[r]^-\alpha&\circ&\circ\ar[l]_-\beta&\circ\ar[l]_-\gamma}
\right)
\bigg/
\langle \gamma\beta \rangle.
$$

\item $B$ is a Nakayama algebra having no 
$$
\mathbf{B}_3=k\left(
\xymatrix@C=1cm@R=0.5cm{
\circ\ar[r]^-\alpha&\circ\ar[r]^-\beta&\circ\ar[r]^-\gamma&\circ}
\right)
\bigg/
\langle\alpha\beta\gamma\rangle
$$
as a quotient algebra.
\end{itemize}
\end{enumerate}
\end{theorem}
\begin{proof}
According to Proposition \ref{prop::local-A2}, Proposition \ref{prop::local-A3} and Proposition \ref{prop::local-N(m)}, it suffices to consider the cases with $|B|\ge 4$ and $B\not\simeq N(m)$. Then, $B$ contains $\mathbf{A}_3^{(++)}$ or $\mathbf{A}_3^{(+-)}$ as a quotient algebra and $A\otimes B$ is representation-infinite if $n\ge 3$, by Proposition \ref{prop::local-A3}. We then assume $n=2$ in the following.

Note that $\widetilde{B}=B$. It is clear that each idempotent truncation of $\widetilde{A}\otimes B$ is isomorphic to some idempotent truncation of $N(s)\otimes B$, for some $s\ge 2$. If $B$ is a Nakayama algebra, then it is mentioned in Theorem \ref{theorem-2} (2) that $N(s)\otimes B$ with $s\ge 4$ is representation-finite if and only if $B$ does not contain $\mathbf{B}_3$ as a quotient algebra. If $B$ is not a Nakayama algebra, then it is shown in Theorem \ref{theorem-2} (3) that $N(s)\otimes B$ with $s\ge 3$ is representation-finite if and only if $B$ is isomorphic to $\mathbf{B}_5$ or $\mathbf{B}_5^{\op}$. In both cases, it is easy to check the locally representation-finiteness of $\widetilde{A}\otimes B$, we omit the details.
\end{proof}

\section{Two-point algebras}
We consider in this section the case $A\otimes B$ with $A,B\not\simeq \mathbf{A}_2$ and $|A|=2$. In fact, the representation-finiteness of $\mathbf{A}_2\otimes C$ for any algebra $C$ is completely determined in \cite{LS-tame-triangular-matrix} as we mentioned in Theorem \ref{theorem-1}. According to the classification of maximal representation-finite two-point algebras in \cite{BG-covering}, the quiver of $A$ is one of the following quivers:
$$
\vcenter{\xymatrix@C=0.8cm{\circ \ar@<0.5ex>[r]^{\ }&\circ \ar@<0.5ex>[l]^{\ }}}\qquad
\vcenter{\xymatrix@C=0.8cm{\circ \ar[r]^{\ }& \circ \ar@(ur,dr)^{\ }}}\quad 
\vcenter{\xymatrix@C=0.8cm{\circ \ar@(dl,ul)^{\ }\ar[r]^{\ }& \circ \ar@(ur,dr)^{\ }}}\quad 
\vcenter{\xymatrix@C=0.8cm{\circ \ar@<0.5ex>[r]^{\ }\ar@(dl,ul)^{\ }&\circ \ar@<0.5ex>[l]^{\ }\ar@(ur,dr)^{\ }}}
$$
Using Theorem \ref{theo::no-cycle} and Theorem \ref{theo::B-type-D}, we assume that the underlying graph of the quiver $Q_B$ of $B$ is of type $\mathbb{A}$, and $|B|\ge 3$.

We have the following easy observation, and this will reduce the general problem to a specific case, i.e., the case $Q_A: \vcenter{\xymatrix@C=0.8cm{\circ \ar@<0.5ex>[r]^{\ }&\circ \ar@<0.5ex>[l]^{\ }}}$.
\begin{proposition}
If the quiver of $A$ contains 
$$
\vcenter{\xymatrix@C=0.8cm{\circ \ar[r]^{\ }& \circ \ar@(ur,dr)^{\ }}}
$$
as a subquiver, then $A\otimes N(3)$ is representation-infinite.
\end{proposition}
\begin{proof}
It is easy to check that the separated quiver of $A\otimes N(3)$ contains 
$$
\vcenter{\xymatrix@C=0.7cm@R=0.1cm{
-&+\ar[l]\ar[dr]&&&\\
&&-&+\ar[l]\ar[r]&-\\
-&+\ar[l]\ar[ur]&&&
}}
$$
as a subquiver, which gives a tame concealed algebra of type $\widetilde{\mathbb{E}}_6$.
\end{proof}

Recall that 
$$
\Delta_2=\xymatrix@C=0.8cm{1\ar@<0.5ex>[r]^{\alpha }&2 \ar@<0.5ex>[l]^{\beta }}.
$$
The covering of $\Delta_2$ is given by
$$
\xymatrix@C=0.7cm@R=0.5cm{
\cdots\ar[r]&1\ar[r]^-\alpha&2 \ar[r]^-\beta&1\ar[r]^-\alpha&2 \ar[r]^-\beta&1\ar[r]&\cdots}.
$$

\begin{theorem}\label{theo::NO(2)-A}
Let $Q_A=\Delta_2$. Suppose the quiver $Q_B$ of $B$ is of type $\mathbb{A}$ and $|B|\ge 3$. Then, $A\otimes B$ is representation-finite if and only if one of the followings hold:
\begin{itemize}
\item $A\simeq k\Delta_2/\langle \beta\alpha\rangle$ and $B\simeq N(n)$ for some $n\ge 3$.

\item $A\simeq N^\circ(2)$ and $B$ is isomorphic to one of $\mathbf{A}_3^{(++)}$, $\mathbf{A}_3^{(+-)}$, $\mathbf{A}_3^{(-+)}$, $\mathbf{B}_5$, $\mathbf{B}_5^{\op}$, or $B$ is a Nakayama algebra having no $\mathbf{B}_3$ as a quotient algebra.
\end{itemize}
\end{theorem}
\begin{proof}
If $\alpha\beta\neq 0$ and $\beta\alpha\neq 0$, then $\widetilde{A}$ contains $\mathbf{B}_1$ as a quotient and $\mathbf{B}_1\otimes N(3)$ is not locally representation-finite, by Theorem \ref{theorem-2} (2). Hence, $A\otimes B$ is representation-infinite.

If $\alpha\beta\neq 0$ and $\beta\alpha=0$, i.e., $A\simeq k\Delta_2/\langle \beta\alpha\rangle$, then the universal Galois covering $\widetilde{A}$ of $A$ contains $\mathbf{A}_3^{(++)}$ as a quotient. If $B\not\simeq N(n)$, then $B$ must contain $\mathbf{A}_3^{(++)}$ or $\mathbf{A}_3^{(+-)}$ or $\mathbf{A}_3^{(-+)}$ as a quotient, and it follows from Proposition \ref{prop::3*3} that $\widetilde{A}\otimes B$ is not locally representation-finite. If $B\simeq N(n)$ for some $n\ge 3$, then $A\otimes B$ is representation-finite since $\widetilde{A}$ contains no $\mathbf{B}_3$ as a quotient, see Theorem \ref{theorem-2} (2).

Suppose $A\simeq N^\circ(2)$. Since $N^\circ(2)$ shares the same universal Galois covering with $N^\circ(1)$, the result follows from Theorem \ref{theo::local-A}.
\end{proof}

\section{Nakayama algebras}
In this section, we consider the case $A\otimes B$ with $|A|\ge 3$ and $|B|\ge 3$. If both $A$ and $B$ are not Nakayama algebras with radical square zero, then one may construct a surjection from $A\otimes B$ to $\mathbf{A}_3^\epsilon\otimes \mathbf{A}_3^\omega$ for some choices of $\epsilon$ and $\omega$ (see Proposition \ref{prop::Nakayama}), such that $A\otimes B$ is representation-infinite according to Proposition \ref{prop::3*3}. Hence, we may assume that one of $A$ and $B$ is isomorphic to $N(n)$ or $N^\circ(n)$, for $3\le n\in \mathbb{N}$.

\subsection{Self-injective Nakayama algebra}
If one of $A$ and $B$ is isomorphic to $N^\circ(n)$ for some $n\ge 3$, say, $A\simeq N^\circ(n)$, then the quiver of $B$ should be of type $\mathbb{A}$; otherwise, $A\otimes B$ is representation-infinite by Theorem \ref{theo::no-cycle} and Theorem \ref{theo::B-type-D}. We have the following result.

\begin{proposition}\label{prop::NO(n)-A}
Let $A\simeq N^\circ(n)$ with some $n\ge 3$. Suppose $|B|\ge 3$ and the quiver of $B$ is of type $\mathbb{A}$. Then, $A\otimes B$ is representation-finite if and only if $B$ is isomorphic to one of $\mathbf{A}_3^{(++)}$, $\mathbf{A}_3^{(+-)}$, $\mathbf{A}_3^{(-+)}$, $\mathbf{B}_5$, $\mathbf{B}_5^{\op}$, or $B$ is a Nakayama algebra having no $\mathbf{B}_3$ as a quotient.
\end{proposition}
\begin{proof}
Since the universal Galois covering of $N^\circ(n)$ coincides with that of $k[x]/(x^2)$, the statement follows from Theorem \ref{theo::local-A}.
\end{proof}

\subsection{Hereditary Nakayama algebra}
If one of $A$ and $B$ is isomorphic to $N(n)$ for some $n\ge 3$, say, $B\simeq N(n)$, then the quiver of $A$ is of type $\mathbb{A}$ or $\widetilde{\mathbb{A}}$ using Theorem \ref{theo::B-type-D}; otherwise, $A\otimes B$ will be representation-infinite. If the quiver of $A$ is of type $\mathbb{A}$ or $A$ is isomorphic to the algebra given by 
\begin{equation}\label{quiver-commutative}
\vcenter{\xymatrix@C=0.7cm@R=0.1cm{
&1\ar[r]&2 \ar[r]&\cdots \ar[r] &s\ar[dr] &\\
\circ \ar[ur] \ar[dr]\ar@{.}[rrrrr] &&&&&\circ \\
&1' \ar[r]  &2' \ar[r] &\cdots \ar[r] &t'\ar[ur] &}}, \quad  1\le s,t \in \mathbb{N},
\end{equation}
with the unique commutativity relation, then $A$ is simply connected and the representation-finiteness of $A\otimes B$ is completely determined in Theorem \ref{theorem-2}. Hence, we only deal with the type $\widetilde{\mathbb{A}}$ except for \eqref{quiver-commutative} in this subsection.

\begin{proposition}\label{prop::N(n)-O(3)}
Let $B\simeq N(n)$ with $n\ge 3$. Suppose the quiver $Q_A$ of $A$ is given by 
$$
\Delta_3: \vcenter{\xymatrix@C=0.5cm@R=0.5cm{
&2\ar[rd]^-\beta&\\
1\ar[ur]^-\alpha&&3\ar[ll]^-\gamma}} \quad \text{or} \quad \Delta_3': \vcenter{\xymatrix@C=0.5cm@R=0.5cm{
&2\ar[rd]^-\beta&\\
1\ar[rr]_-\gamma\ar[ur]^-\alpha&&3}}.
$$
Then, $A\otimes B$ is representation-finite if and only if $A$ is isomorphic $k\Delta_3/\langle \beta\gamma, \gamma\alpha \rangle$ or $N^\circ(3)$.
\end{proposition}
\begin{proof}
(1) Suppose $Q_A=\Delta_3$. The covering of $Q_A$ is given by 
$$
\xymatrix@C=1cm@R=0.5cm{
\cdots\ar[r]&3\ar[r]^-\gamma&1\ar[r]^-\alpha&2\ar[r]^-\beta&3\ar[r]^-\gamma&1\ar[r]^-\alpha&2\ar[r]&\cdots }.
$$
If $A$ contains $\mathbf{C}_1:=k\Delta_3/\langle \alpha\beta\gamma, \gamma\alpha \rangle$ as a quotient algebra, then $\widetilde{A}$ contains 
$$
\mathbf{B}_2=k\left(
\xymatrix@C=0.75cm@R=0.5cm{
\circ\ar[r]^-\alpha&\circ\ar[r]^-\beta&\circ\ar[r]^-\gamma&\circ\ar[r]^-\delta&\circ\ar[r]^-\sigma &\circ}
\right)
\bigg/
\langle \alpha\beta\gamma, \gamma\delta \rangle
$$
as a quotient algebra. It is shown in Theorem \ref{theorem-2} (2) that $\mathbf{B}_2\otimes N(3)$ is representation-infinite, and hence, $\widetilde{A}\otimes N(n)$ is not locally representation-finite for any $n\ge 3$. 

If $A$ does not contain $\mathbf{C}_1$ as a quotient algebra, we have $a_1a_2=0, a_2a_3\neq 0$, or $a_1a_2\neq 0, a_2a_3= 0$, or $a_1a_2=a_2a_3=0$, in each subquiver 
$$
\xymatrix@C=1cm@R=0.5cm{
i_1 \ar[r]^-{a_1}&i_2\ar[r]^-{a_2}&i_3\ar[r]^-{a_3}&i_1}
$$
of $\widetilde{A}$, such that $\widetilde{A}$ contains no $\mathbf{B}_3$ as a quotient. Each idempotent truncation of $\widetilde{A}\otimes N(n)$ is known to be representation-finite by Theorem \ref{theorem-2} (2), such that $\widetilde{A}\otimes N(n)$ is locally representation-finite for any $n\ge 3$. In this case, $A$ is isomorphic to either $k\Delta_3/\langle \beta\gamma, \gamma\alpha \rangle$ or $N^\circ(3)$.

(2) If $Q_A=\Delta_3'$, then $A$ must be isomorphic to $k\Delta_3'/\langle \alpha\beta\rangle$; otherwise, $A$ is a representation-infinite hereditary algebra. The separated quiver of $A\otimes N(3)$ contains 
$$
\xymatrix@C=0.7cm@R=0.1cm{
+\ar[dr]&&&&&&+\ar[dl]\\
&-&+\ar[r]\ar[l]&-&+\ar[r]\ar[l]&-&\\
+\ar[ur]&&&&&&+\ar[ul]
}
$$
as a subquiver, and hence, $A\otimes N(3)$ is representation-infinite by Proposition \ref{prop::rad-square-zero}.
\end{proof}

\begin{proposition}\label{prop::N(n)-O(4)}
Let $B\simeq N(n)$ with $n\ge 3$. Suppose the quiver $Q_A$ of $A$ is given by 
$$
\Delta_4: \vcenter{\xymatrix@C=0.8cm@R=0.6cm{
1\ar[r]^-\alpha&2\ar[d]^-\beta\\
4\ar[u]^-\delta&3\ar[l]^-\gamma}} 
\quad \text{or} \quad 
\Delta_4': \vcenter{\xymatrix@C=0.8cm@R=0.6cm{
1\ar[r]^-\alpha\ar[d]_-\delta&2\ar[d]^-\beta\\
4&3\ar[l]^-\gamma}} 
\quad \text{or} \quad 
\Delta_4'': \vcenter{\xymatrix@C=0.8cm@R=0.6cm{
1\ar[r]^-\alpha\ar[d]_-\delta&2\ar[d]^-\beta\\
4\ar[r]_-\gamma&3}} 
\quad \text{or} \quad 
\Delta_4''': \vcenter{\xymatrix@C=0.8cm@R=0.6cm{
1\ar[r]^-\alpha\ar[d]_-\delta&2\\
4&3\ar[l]^-\gamma\ar[u]_-\beta}} 
$$
Then, $A\otimes B$ is representation-finite if and only if one of the followings hold:
\begin{enumerate}
\item[\rm{(1)}] $n=3$ and $A$ is isomorphic to one of $k\Delta_4/\langle \alpha\beta\gamma, \gamma\delta, \delta\alpha \rangle$.

\item[\rm{(2)}] $A$ is isomorphic to one of $k\Delta_4/\langle \beta\gamma, \delta\alpha \rangle$, $k\Delta_4/\langle \beta\gamma, \gamma\delta, \delta\alpha \rangle$ and $N^\circ(4)$.
\end{enumerate}
\end{proposition}
\begin{proof}
If $Q_A=\Delta_4'$, then $A$ contains $\mathbf{A}_4^{(-+-)}$ as a quotient, and $A\otimes N(3)$ is representation-infinite by Theorem \ref{theorem-2} (2). If $Q_A=\Delta_4''$, then $A\otimes \mathbf{A}_2$ contains a tame concealed algebra of type $\widetilde{\mathbb{A}}_5$ as a quotient, i.e., the bullets in the following quiver
$$
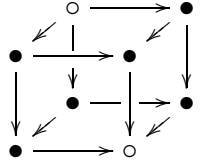
\vcenter{\xymatrix@C=0.3cm@R=0.2cm{& \circ \ar[rr]\ar'[d][dd] \ar[dl]& & \bullet \ar[dd]\ar[dl]\\ \bullet \ar[rr]\ar[dd]& & \bullet \ar[dd]\\& \bullet \ar[dl]\ar'[r][rr]& & \bullet \ar[dl]\\ \bullet\ar[rr]& & \circ}}.
$$
If $Q_A=\Delta_4'''$, then $A$ is a representation-infinite hereditary algebra. In all three cases, $A\otimes B$ is representation-infinite.

Suppose $Q_A=\Delta_4$. The covering of $Q_A$ is given by 
$$
\xymatrix@C=0.7cm@R=0.5cm{
\cdots\ar[r]&4\ar[r]^-\delta&1\ar[r]^-\alpha&2\ar[r]^-\beta&3\ar[r]^-\gamma&4\ar[r]^-\delta&1\ar[r]^-\alpha&2\ar[r]^-\beta&3\ar[r] &\cdots }.
$$
\begin{itemize}
\item $n=3$. If $\rad^3 A \neq 0$, then $A$ contains $\mathbf{A}_4^{(+++)}$ as a quotient such that $A\otimes N(3)$ is obviously representation-infinite. Suppose $\rad^3 A=0$. If $A$ contains $\mathbf{C}_2:=k\Delta_4/\langle \alpha\beta\gamma, \beta\gamma\delta, \delta\alpha \rangle$ as a quotient, then $\widetilde{A}$ contains $\mathbf{B}_2$ as a quotient. It is mentioned in Theorem \ref{theorem-2} (2) that $\mathbf{B}_2\otimes N(3)$ is representation-infinite, and hence, $\widetilde{A}\otimes N(3)$ is not locally representation-finite. 

If $A$ contains no $\mathbf{C}_2$ as a quotient (and $\rad^3 A=0$), then at least one of $a_1a_2=0$, $a_2a_3=0$, $a_3a_4=0$ holds in each subquiver 
\begin{center}
$\xymatrix@C=1cm@R=0.5cm{
i_1 \ar[r]^-{a_1}&i_2\ar[r]^-{a_2}&i_3\ar[r]^-{a_3}&i_4\ar[r]^-{a_4} &i_1}$
\end{center}
of $\widetilde{A}$, such that $\widetilde{A}$ contains no 
\begin{center}
$\mathbf{B}_1=k\left(
\xymatrix@C=1cm@R=0.5cm{
\circ\ar[r]^-\alpha&\circ\ar[r]^-\beta&\circ\ar[r]^-\gamma&\circ\ar[r]^-\delta&\circ}
\right)
\bigg/
\langle \alpha\beta\gamma, \beta\gamma\delta \rangle$
\end{center}
as a quotient; if moreover, $a_1a_2\neq 0$ and $a_2a_3\neq 0$ in the following subquiver
\begin{center}
$\xymatrix@C=1cm@R=0.5cm{
i_1 \ar[r]^-{a_1}&i_2\ar[r]^-{a_2}&i_3\ar[r]^-{a_3}&i_4\ar[r]^-{a_4} &i_1 \ar[r]^-{a_1}&i_2}$
\end{center}
of $\widetilde{A}$, then we must have $a_3a_4=a_4a_1=0$ such that $\widetilde{A}$ contains no $\mathbf{B}_2$ as a quotient. It then follows from Theorem \ref{theorem-2} (2) that each idempotent truncation of $\widetilde{A}\otimes N(3)$ is representation-finite. We conclude that $\widetilde{A}\otimes N(3)$ is locally representation-finite. In this case, $A$ is isomorphic to one of $k\Delta_4/\langle \alpha\beta\gamma, \gamma\delta, \delta\alpha \rangle$, $k\Delta_4/\langle \beta\gamma, \delta\alpha \rangle$, $k\Delta_4/\langle \beta\gamma, \gamma\delta, \delta\alpha \rangle$ and $N^\circ(4)$.

\item $n\geq 4$. If $A$ contains
\begin{center}
$\mathbf{B}_3=k\left(
\xymatrix@C=1cm@R=0.5cm{
\circ\ar[r]^-\alpha&\circ\ar[r]^-\beta&\circ\ar[r]^-\gamma&\circ}
\right)
\bigg/
\langle\alpha\beta\gamma\rangle$
\end{center}
as a quotient, then $N(n)\otimes \mathbf{B}_3$ is representation-infinite for any $n\geq 4$, mentioned in Theorem \ref{theorem-2} (2). If $A$ contains no $\mathbf{B}_3$ as a quotient, then we have $a_1a_2=0, a_2a_3\neq 0$, or $a_1a_2\neq 0, a_2a_3= 0$, or $a_1a_2=a_2a_3=0$, in each subquiver 
\begin{center}
$\xymatrix@C=1cm@R=0.5cm{
i_1\ar[r]^-{a_1}&i_2\ar[r]^-{a_2}&i_3\ar[r]^-{a_3}&i_4}$
\end{center}
of $\widetilde{A}$, such that $\widetilde{A}$ contains no $\mathbf{B}_3$ as a quotient. It follows from Theorem \ref{theorem-2} (2) again that each idempotent truncation of $\widetilde{A}\otimes N(n)$ is representation-finite, and then, $\widetilde{A}\otimes N(n)$ is locally representation-finite for any $n\ge 4$. In this case, $A$ is isomorphic to one of $k\Delta_4/\langle \beta\gamma, \delta\alpha \rangle$, $k\Delta_4/\langle \beta\gamma, \gamma\delta, \delta\alpha \rangle$ and $N^\circ(4)$.
\end{itemize}
We have completed the proof.
\end{proof}

\begin{proposition}\label{prop::N(n)-O(5)}
Let $|A|=5$ and $B\simeq N(n)$ with $n\ge 3$. Suppose the quiver of $A$ is of type $\widetilde{\mathbb{A}}$ and $A$ is not the algebra in \eqref{quiver-commutative}.
Then, $A\otimes B$ is representation-finite if and only if  one of the followings hold:
\begin{itemize}
\item $n=3$ and $A$ is isomorphic to $k\Delta_5/\langle \alpha\beta\gamma, \gamma\delta, \delta\sigma, \sigma\alpha \rangle$, where
\begin{center}
$\Delta_5: \xymatrix{1 \ar[r]_\alpha & 2 \ar[r]_\beta &3 \ar[r]_\gamma &4 \ar[r]_\delta & 5 \ar@/_1pc/[llll]_\sigma}$.
\end{center}

\item $A$ is isomorphic to one of $k\Delta_5/\langle \beta\gamma, \gamma\delta, \sigma\alpha \rangle$, $k\Delta_5/\langle \beta\gamma, \gamma\delta, \delta\sigma, \sigma\alpha \rangle$ and $N^\circ(5)$.
\end{itemize}
\end{proposition}
\begin{proof}
If $Q_A$ admits a sink, then $Q_A$ contains 
$$
\xymatrix{\circ\ar@{-}[r] &\circ\ar[r]  &\circ &\circ\ar[l] &\circ\ar@{-}[l]}
$$
as a subquiver. If $A$ contains $\mathbf{A}_4^\epsilon$ for some $\epsilon$ as a quotient, then $A\otimes N(3)$ is known to be representation-infinite. Otherwise, $A$ contains 
\begin{center}
$\mathbf{B}_7=k\left(
\xymatrix@C=0.85cm@R=0.5cm{
\circ\ar[r]^-\delta &\circ\ar[r]^-\alpha&\circ&\circ\ar[l]_-\beta&\circ\ar[l]_-\gamma}
\right)
\bigg/
\langle \gamma\beta, \delta\alpha \rangle$
\end{center}
as a quotient, and $\mathbf{B}_7\otimes N(3)$ is again representation-infinite by Theorem \ref{theorem-2} (3).

If $Q_A$ admits no sink, then it is an orientated cycle, i.e., $Q_A=\Delta_5$.
\begin{itemize}
\item $n=3$. We may assume $\rad^3 A=0$. It is known from Theorem \ref{theorem-2} (2) that $N(3)\otimes \mathbf{B}_1$ is representation-infinite, where
\begin{center}
$\mathbf{B}_1=k\left(
\xymatrix@C=1cm@R=0.5cm{
\circ\ar[r]^-\alpha&\circ\ar[r]^-\beta&\circ\ar[r]^-\gamma&\circ\ar[r]^-\delta&\circ}
\right)
\bigg/
\langle \alpha\beta\gamma, \beta\gamma\delta \rangle$
\end{center}
If $A$ contains $\mathbf{C}_3:=k\Delta_5/\langle \alpha\beta\gamma, \gamma\delta, \sigma\alpha \rangle$ as a quotient algebra, then the universal Galois covering $\widetilde{A}$ of $A$ contains 
\begin{center}
$\mathbf{B}_2=k\left(
\xymatrix@C=0.75cm@R=0.5cm{
\circ\ar[r]^-\alpha&\circ\ar[r]^-\beta&\circ\ar[r]^-\gamma&\circ\ar[r]^-\delta&\circ\ar[r]^-\sigma &\circ}
\right)
\bigg/
\langle \alpha\beta\gamma, \gamma\delta \rangle$
\end{center}
as a quotient; since $\mathbf{B}_2\otimes N(3)$ is representation-infinite following Theorem \ref{theorem-2} (2), we conclude that $\widetilde{A}\otimes N(3)$ is not locally representation-finite.

If $A$ contains no $\mathbf{B}_1$ as a quotient (and $\rad^3 A=0$), then $\widetilde{A}$ contains no $\mathbf{B}_1$ as a quotient; if moreover, $A$ contains no $\mathbf{C}_3$ as a quotient, then $a_1a_2\neq 0$ and $a_2a_3\neq 0$ in the following subquiver
\begin{center}
$\xymatrix@C=1cm@R=0.5cm{
i_1 \ar[r]^-{a_1}&i_2\ar[r]^-{a_2}&i_3\ar[r]^-{a_3}&i_4\ar[r]^-{a_4} &i_5 \ar[r]^-{a_5}&i_1}$
\end{center}
of $\widetilde{A}$ must imply $a_3a_4=a_4a_5=0$, such that $\widetilde{A}$ contains no $\mathbf{B}_2$ as a quotient. We then use Theorem \ref{theorem-2} (2) to deduce that $\widetilde{A}\otimes N(3)$ is locally representation-finite. In this case, $A$ is isomorphic to one of $k\Delta_5/\langle \alpha\beta\gamma, \gamma\delta, \delta\sigma, \sigma\alpha \rangle$, $k\Delta_5/\langle \beta\gamma, \gamma\delta, \sigma\alpha \rangle$, $k\Delta_5/\langle \beta\gamma, \gamma\delta, \delta\sigma, \sigma\alpha \rangle$ and $N^\circ(5)$.

\item $n\ge 4$. If $A$ contains no $\mathbf{B}_3$ as a quotient, then it is easy to show that $\widetilde{A}\otimes N(n)$ is locally representation-finite, we omit the details. In this case, $A$ is isomorphic to one of $k\Delta_5/\langle \beta\gamma, \gamma\delta, \sigma\alpha \rangle$, $k\Delta_5/\langle \beta\gamma, \gamma\delta, \delta\sigma, \sigma\alpha \rangle$ and $N^\circ(5)$.
\end{itemize}
We have finished the proof.
\end{proof}

\begin{theorem}\label{theo::NO(n)-N(n)}
Let $B\simeq N(n)$ for some $n\ge 3$. Suppose the quiver of $A$ is of type $\widetilde{\mathbb{A}}$ and $A$ is not the algebra in \eqref{quiver-commutative}. Then, $A\otimes B$ is representation-finite if and only if the quiver of $A$ is an oriented cycle, and one of the following holds.
\begin{enumerate}
\item[\rm{(1)}] $A$ is isomorphic one of $k\Delta_3/\langle \beta\gamma, \gamma\alpha \rangle$,  $k\Delta_4/\langle \beta\gamma, \delta\alpha \rangle$, $k\Delta_4/\langle \beta\gamma, \gamma\delta, \delta\alpha \rangle$, $k\Delta_5/\langle \beta\gamma, \gamma\delta, \sigma\alpha \rangle$, $k\Delta_5/\langle \beta\gamma, \gamma\delta, \delta\sigma, \sigma\alpha \rangle$, $N^\circ(3)$, $N^\circ(4)$ and $N^\circ(5)$.

\item[\rm{(2)}] $n=3$ and $A$ is isomorphic to $k\Delta_4/\langle \alpha\beta\gamma, \gamma\delta, \delta\alpha \rangle$ or $k\Delta_5/\langle \alpha\beta\gamma, \gamma\delta, \delta\sigma, \sigma\alpha \rangle$.

\item[\rm{(3)}] $n=3$, $|A|\ge 6$, $\rad^3 A=0$ and $A$ does not contain $\mathbf{B}_1$ or $\mathbf{B}_2$ as a quotient algebra.

\item[\rm{(4)}] $n\ge 4$, $|A|\ge 6$ and $A$ does not contain $\mathbf{B}_3$ as a quotient algebra.
\end{enumerate}
\end{theorem}
\begin{proof}
According to Proposition \ref{prop::N(n)-O(3)}, Proposition \ref{prop::N(n)-O(4)} and Proposition \ref{prop::N(n)-O(5)}, we only need to consider the case $|A|\ge 6$. In this case, $A$ contains one of $\mathbf{B}_1$, $\mathbf{B}_2$, $\mathbf{B}_3$ as a quotient if and only if $\widetilde{A}$ does. Since each idempotent truncation of $\widetilde{A}$ is a Nakayama algebra, then the statements follow from Theorem \ref{theorem-2} (2). In particular, the case $A\simeq N^\circ(m)$ is already considered in Proposition \ref{prop::NO(n)-A}.
\end{proof}

\vspace{0.5cm}
\section*{Acknowledgements}
The author is partially supported by National Key Research and Development Program of China (Grant No. 2020YFA0713000) and China Postdoctoral Science Foundation (Grant No. YJ20220119 and No. 2023M731988).

\end{document}